\documentclass{amsart}
\usepackage{hyperref}
\usepackage{pstricks, pst-node,pst-tree}
\usepackage{a4wide}
 
\title{Bijections on rooted trees with fixed size of maximal decreasing subtrees}
\author{Jang Soo Kim}
\email{kimjs@math.umn.edu}
\subjclass[2000]{05A15, 05A05}
\keywords{rooted trees, maximal decreasing subtrees}
\date{\today}

\newtheorem{thm}{Theorem}[section]
\newtheorem{lem}[thm]{Lemma}

\newtheorem{cor}[thm]{Corollary}
\theoremstyle{definition}
\newtheorem{example}{Example}

\theoremstyle{remark}

\DeclareMathOperator{\MD}{MD}
\newcommand\A{\mathcal{A}}
\newcommand\F{\mathcal{F}}
\newcommand\X{\mathcal{X}}
\newcommand\Y{\mathcal{Y}}

\newcommand\FF{\mathfrak{F}}
\newcommand\TT{\mathfrak{T}}
\renewcommand\AA{\mathfrak{A}}
\newcommand\BB{\mathfrak{B}}
\newcommand\CC{\mathfrak{C}}

\def\sch.{Schr{\"o}der}

\psset{nodesep=2pt,levelsep=20pt}
\def\comma{\raisebox{-4pt}{\mbox{ ,\quad}}}

\begin{document}

\begin{abstract}
  Seo and Shin showed that the number of rooted trees on $[n+1]$ such that the
  maximal decreasing subtree with the same root has $k+1$ vertices is equal to
  the number of functions $f:[n]\to[n]$ such that the image of $f$ contains
  $[k]$. We give a bijective proof of this theorem.
\end{abstract}

\maketitle

\section{Introduction}\label{sec:introduction}

A \emph{tree} on a finite set $X$ is a connected acyclic graph with vertex set
$X$. A \emph{rooted tree} is a tree with a distinguished vertex called a
\emph{root}. It is well-known that the number of rooted trees on
$[n]=\{1,2,\dots,n\}$ is $n^{n-1}$, see \cite[5.3.2~Proposition]{EC2}.

Suppose $T$ is a rooted tree with root $r$. For a vertex $v\ne r$ of $T$ there
is a unique path $(u_1,u_2,\dots,u_i)$ from $r=u_1$ to $v=u_i$. Then $u_{i-1}$
is called the \emph{parent} of $v$, and $v$ is called a \emph{child} of
$u_{i-1}$.  For two vertices $u$ and $v$, we say that $u$ is a \emph{descendant}
of $v$ if the unique path from $r$ to $u$ contains $v$. Note that every vertex
is a descendant of itself. A \emph{leaf} is a vertex with no children. A rooted
tree is \emph{decreasing} if every nonleaf is greater than its children.  The
\emph{maximal decreasing subtree} of $T$, denoted $\MD(T)$, is the maximal
subtree such that it has the same root as $T$ and it is decreasing. If the root
of $T$ has no smaller children, $T$ is called \emph{minimally rooted}. 

The notion of maximal decreasing subtree was first appeared in \cite{Chauve2000}
in order to prove the following theorem.
\begin{thm}\cite[Theorem~2.1]{Chauve2000}\label{thm:CDG}
The number of rooted trees on $[n+1]$
such that the root has $\ell$ smaller children equals $\binom{n}{\ell} n^{n-\ell}$.
\end{thm}

Recently, maximal decreasing subtrees reappeared in the study of a certain free
Lie algebra over rooted trees by Bergeron and Livernet \cite{Bergeron2010}. Seo
and Shin \cite{SeoShin} proved some enumeration properties of rooted trees with
fixed size of maximal decreasing subtrees.



We denote by $\TT_{n,k}$ the set of rooted trees on $[n+1]$ whose maximal
decreasing subtrees have $k+1$ vertices. Let $\FF_{n,k}$ denote the set of
functions $f:[n]\to[n]$ such that $[k]\subset f([n])$, where $[0]=\emptyset$.
Equivalently, $\FF_{n,k}$ is the set of words on $[n]$ of length $n$ such that
each of $1,2,\dots,k$ appears at least once.  Using the Pr\"ufer code one can
easily see that $\FF_{n,k}$ is in bijection with the set of rooted trees on
$[n+1]$ such that $n+1$ is a leaf and $1,2,\dots,k$ are nonleaves.  Thus, we
will consider $\FF_{n,k}$ as the set of such trees.

Seo and Shin \cite{SeoShin} proved the following theorem. 

\begin{thm}\cite{SeoShin}\label{thm:SeoShin}
We have
\[
|\TT_{n,k}| = |\FF_{n,k}|.
\]
\end{thm}

In \cite{SeoShin} they showed Theorem~\ref{thm:SeoShin} by finding formulas for
both sides and computing the formulas. In this paper we provide a bijective
proof Theorem~\ref{thm:SeoShin}, which consists of several bijections between
certain objects, see Theorem~\ref{thm:main}. In order to state the objects in
Theorem~\ref{thm:main} we need the following definitions.

An \emph{ordered forest} on a finite set $X$ is an ordered tuple of rooted trees
whose vertex sets form a partition of $X$.  We say that an ordered forest
$(T_0,T_1,\dots,T_\ell)$ is \emph{$k$-good} if it satisfies the following
conditions:
\begin{enumerate}
\item If $\ell=0$, then $T_0$ has only one vertex $v$ and we have $v\in [k]$.
\item If $\ell\geq1$, then $T_1,T_2,\dots,T_\ell$ are minimally rooted, and the
  number of vertices of $T_0,T_1,\dots,T_i$ contained in $[k]$ is at least $i+1$
  when $i=0,1,2,\dots,\ell-1$, and equal to $\ell$ when $i=\ell$.
\end{enumerate}

We now state the main theorem of this paper.

\begin{thm}\label{thm:main}
  The following sets have the same cardinality:
  \begin{enumerate}
\item the set $\TT_{n,k}$ of rooted trees on $[n+1]$ whose maximal decreasing
subtrees have $k+1$ vertices,
\item the set $\AA_{n,k}$ of cycles of $k+1$ minimally rooted trees such that the
  vertex sets of the trees form a partition of $[n+1]$,
\item the set $\BB_{n,k}$ of ordered forests on $[n]$ such that the last $k$ trees
  are minimally rooted, 
\item the set $\CC_{n,k}$ of sequences of $k$-good ordered forests such that the
  vertex sets of the forests form a partition of $[n]$,
\item the set $\FF_{n,k}$ of rooted trees on $[n+1]$ such that $n+1$ is a leaf,
  and $1,2,\dots,k$ are nonleaves. 
\end{enumerate}
\end{thm}

In Section~\ref{sec:bijections} we find bijections proving
Theorem~\ref{thm:main}. The ideas in the bijections have some applications.  In
Section~\ref{sec:some-properties-tt_n} we find a bijective proof of the
following identity, which (finding a bijective proof) is stated as an open
problem in \cite{SeoShin}:
\[
    \sum_{k\ge0} \frac{1}{k} |\TT_{n,k}| = n^n.
\]
In Section~\ref{sec:anoth-proof-theor}, we gives another bijective proof of
Theorem~\ref{thm:CDG}.

From now on all trees in this paper are rooted trees.

\section{Bijections}\label{sec:bijections}

In this section we will find four bijections to prove Theorem~\ref{thm:main}. We
assume that $n$ and $k$ are fixed nonnegative integers.  We will write cycles
using brackets to distinguish them from sequences. For instance, $[a_1,a_2,a_3]$
is a cycle and $(a_1,a_2,a_3)$ is a sequence, thus $[a_1,a_2,a_3]=[a_2,a_3,a_1]$
and $(a_1,a_2,a_3)\ne(a_2,a_3,a_1)$. For a tree or a forest $T$, we denote by
$V(T)$ the set of vertices in $T$.

\subsection{A bijection $\alpha:\TT_{n,k}\to\AA_{n,k}$}

We will explain the map $\alpha$ by an example.  Let $T\in\TT_{19,7}$ be the
following tree.
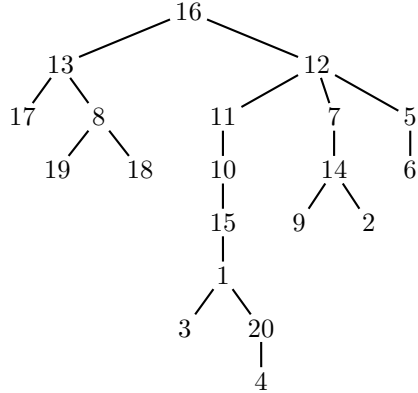
\[
\pstree{\TR{16}}{
  \pstree{\TR{13}}{
    \TR{17}
    \pstree{\TR{8}}{
      \TR{19}
      \TR{18}
    }
  }
  \pstree{\TR{12}}{
    \pstree{\TR{11}}{
      \pstree{\TR{10}}{
        \pstree{\TR{15}}{
          \pstree{\TR{1}}{
            \TR{3}
            \pstree{\TR{20}}{
              \TR{4}
            }
          }
        }
      }
    }
    \pstree{\TR{7}}{
      \pstree{\TR{14}}{
        \TR{9} \TR{2}
      }
    }
    \pstree{\TR{5}}{
      \TR{6}
    }
  }
}
\]
Then we can decompose $T$ as follows:
\begin{equation}
  \label{eq:1}
T \Leftrightarrow
\left(
\raisebox{1.5cm}{
\pstree{\TR{16}}{
  \pstree{\TR{13}}{
    \TR{8}
  }
  \pstree{\TR{12}}{
    \pstree{\TR{11}}{
      \TR{10}
    }
    \TR{7}
    \TR{5}
  }
} 
\comma
\pstree{\TR{13}}{
  \TR{17}
}
\comma
\pstree{\TR{10}}{
  \pstree{\TR{15}}{
    \pstree{\TR{1}}{
      \TR{3}
      \pstree{\TR{20}}{
        \TR{4}
      }
    }
  }
}
\comma
    \pstree{\TR{8}}{
      \TR{19}
      \TR{18}
    }
\comma
\pstree{\TR{7}}{\pstree{\TR{14}}{\TR{9} \TR{2}}} \comma
\pstree{\TR{5}}{\TR{6}} 
}
\right),
\end{equation}
where the first tree is $\MD(T)$, and the rest are the trees with more than one
vertex in the forest obtained from $T$ by removing the edges in $\MD(T)$.  We
now construct a cycle $C$ corresponding to $\MD(T)$ as follows. First, let $C$
be the cycle containing only the maximal vertex $m$, which is the root of
$\MD(T)$. For each remaining vertex $v$, starting from the largest vertex to the
smallest vertex, we insert $v$ in $C$ after the parent of $v$. In the current
example, we get the cycle $[16, 12, 5, 7, 11, 10, 13, 8]$. It is easy to see
that this process is invertible. In fact this is equivalent to the well-known
algorithm called the depth-first search (preorder).

For each element $v$ except the largest element in this cycle, if there is a
tree with root $v$ in \eqref{eq:1} replace $v$ with the tree. We then define
$\alpha(T)$ to be the resulting cycle. It is not hard to see that $\alpha$ is a
bijection.  In the current example, we have
\begin{equation}
  \label{eq:2}
\alpha(T) =  \left[
\raisebox{1.5cm}{
\TR{16} \comma
\TR{12} \comma
\pstree{\TR{5}}{\TR{6}} \comma
\pstree{\TR{7}}{\pstree{\TR{14}}{\TR{9} \TR{2}}} \comma
\TR{11}
\comma
\pstree{\TR{10}}{
  \pstree{\TR{15}}{
    \pstree{\TR{1}}{
      \TR{3}
      \pstree{\TR{20}}{
        \TR{4}
      }
    }
  }
}
\comma
\pstree{\TR{13}}{
  \TR{17}
}
\comma
    \pstree{\TR{8}}{
      \TR{19}
      \TR{18}
    }
}
\right] \in \BB_{19,7}.
\end{equation}

\subsection{A bijection $\beta:\AA_{n,k}\to\BB_{n,k}$}

In order to define the map $\beta$ we need two bijections $\phi$ and $\rho$ in
the following two lemmas. These bijections will also be used in other
subsections.

\begin{lem}\cite{Chauve2000}\label{lem:CDG}
  Let $A\subset [n]$.  There is a bijection $\phi$ from the set of minimally
  rooted trees on $A$ to the set of rooted trees on $A$ such that $\max(A)$ is a
  leaf.
\end{lem}
\begin{proof}
  We will briefly describe the bijection $\phi$. See \cite{Chauve2000} for the
  details. Consider a minimally rooted tree $T$ on $A$ with root $r$. Let $P$ be
  the subtree of $T$ rooted at $\max(A)$ containing all descendants of
  $\max(A)$, and let $Q$ be the tree obtained from $T$ by deleting $P$
  (including $\max(A)$).  We now consider the forest obtained from $P$ by
  removing all edges of $\MD(P)$. Suppose this forest has $\ell$ trees
  $T_1,T_2,\dots,T_\ell$ rooted at, respectively, $r_1,r_2,\dots,r_\ell$.  If
  $V(\MD(P)) = \{u_1<u_2<\dots<u_t\}$ and $V(\MD(P))\setminus\{\max(A)\} \cup
  \{r\} = \{v_1<v_2<\dots<v_t\}$, let $T'$ be the tree obtained from $\MD(P)$ by
  replacing $u_i$ with $v_i$ for all $i$. Then $\phi(T)$ is the tree obtained
  from $T'$ by attaching $T_i$ at $r_i$ for $i=1,2,\dots,\ell$ and attaching $Q$
  at $r$.
\end{proof}

\begin{lem}\label{lem:leaf}
  Let $A\subset[n]$.  There is a bijection $\rho$ from the set of rooted trees
  on $A$ such that $\max(A)$ is a leaf to the set of ordered forests on
  $A\setminus\{\max(A)\}$.
\end{lem}
\begin{proof}
  Suppose $T$ is a rooted tree on $A$ such that $\max(A)$ is a leaf.  Let
  $r=r_1, r_2,\dots, r_{\ell+1}=\max(A)$ be the unique path from the root $r$ of
  $T$ to the leaf $\max(A)$.  Let $R_1,R_2,\dots,R_{\ell}$ be the rooted trees
  with roots $r_1,r_2,\dots,r_{\ell}$ respectively in the forest obtained from
  $T$ by removing the edges $r_1r_2,r_2r_3,\dots,r_{\ell}r_{\ell+1}$ and the
  vertex $r_{\ell+1}=\max(A)$. We define $\rho(T)=(R_1,R_2,\dots,R_{\ell})$. It
  is easy to see that $\rho$ is a desired bijection.
\end{proof}

Let $[T_0,T_1,\dots,T_k]\in \AA_{n,k}$. Since $[T_0,T_1,\dots,T_k]$ is a cycle,
we can assume that $n+1\in T_0$.  By Lemmas~\ref{lem:CDG} and \ref{lem:leaf},
the vertex $n+1$ in $\phi(T_0)$ is a leaf, and $\rho(\phi(T_0)) =
(R_1,R_2,\dots,R_{\ell})$ is an ordered forest on $V(T_0)\setminus\{n+1\}$. We
define $\beta([T_0,T_1,\dots,T_k])= (R_1,R_2,\dots,R_{\ell},
T_1,T_2,\dots,T_k)$.  Since both $\phi$ and $\rho$ are invertible, $\beta$ is a
bijection.

\begin{example}
  Let $\F$ be the cycle in \eqref{eq:2}. Then we can write $\F$ as
\[
\F = \left[
\raisebox{1.5cm}{
\pstree{\TR{10}}{
  \pstree{\TR{15}}{
    \pstree{\TR{1}}{
      \TR{3}
      \pstree{\TR{20}}{
        \TR{4}
      }
    }
  }
}
\comma
\pstree{\TR{13}}{
  \TR{17}
}
\comma
    \pstree{\TR{8}}{
      \TR{19}
      \TR{18}
    }
\comma
\TR{16} \comma
\TR{12} \comma
\pstree{\TR{5}}{\TR{6}} \comma
\pstree{\TR{7}}{\pstree{\TR{14}}{\TR{9} \TR{2}}} \comma
\TR{11}
}
\right]. 
\]
Then
\[
(\rho\circ\phi)
\left(
\raisebox{1.5cm}{
\pstree{\TR{10}}{
  \pstree{\TR{15}}{
    \pstree{\TR{1}}{
      \TR{3}
      \pstree{\TR{20}}{
        \TR{4}
      }
    }
  }
}
}
\right) 
= \rho\left(
\raisebox{1.5cm}{
\pstree{\TR{10}}{
  \pstree{\TR{15}}{
    \pstree{\TR{1}}{
      \TR{3}
      \TR{20}
    }
  }
  \TR{4}
}
}
\right)=
\left(
\raisebox{.4cm}{
\pstree{\TR{10}}{\TR{4}} \comma
\TR{15} \comma
\pstree{\TR{1}}{\TR{3}}
}
\right).
\]
Thus
\[
\beta(\F) = \left(
\raisebox{.8cm}{
\pstree{\TR{10}}{\TR{4}} \comma
\TR{15} \comma
\pstree{\TR{1}}{\TR{3}}  \comma
\pstree{\TR{13}}{
  \TR{17}
}
\comma
    \pstree{\TR{8}}{
      \TR{19}
      \TR{18}
    }
\comma
\TR{16} \comma
\TR{12} \comma
\pstree{\TR{5}}{\TR{6}} \comma
\pstree{\TR{7}}{\pstree{\TR{14}}{\TR{9} \TR{2}}} \comma
\TR{11}
}
\right) \in \BB_{19,7}.
\]

\end{example}

\subsection{A bijection $\gamma:\BB_{n,k}\to\CC_{n,k}$}

We call a vertex with label less than or equal to $k$ a \emph{special vertex}.
For two ordered forests $\X$ and $\Y$ whose vertex sets are disjoint and
contained in $[n]$, the pair $(\X,\Y)$ is called a \emph{balanced pair} if the
trees in $\Y$ are minimally rooted and the number of special vertices in $\X$
and $\Y$ is equal to the number of trees in $\Y$.

For two sets $A$ and $B$, the \emph{disjoint union} $A\uplus B$ is just the
union of $A$ and $B$. However, if we write $A\uplus B$, it is always assumed that
$A\cap B=\emptyset$. 

\begin{lem}
  There is a bijection $f$ from the set of balanced pairs $(\X,\Y)$ to the set
  of pairs $(\A,(\X',\Y'))$ of a $k$-good ordered forest $\A$ and a balanced
  pair $(\X',\Y')$ such that $V(\X)\uplus V(\Y) = V(\A) \uplus V(\X') \uplus
  V(\Y')$.
\end{lem}
\begin{proof}
  Suppose $\X=(X_1,X_2,\dots,X_s)$ and $\Y=(Y_1,Y_2,\dots,Y_t)$.  We define
  $f(\X,\Y)=(\A,(\X',\Y'))$ as follows.

Case 1: If $s\geq1$, and $X_1$ does not contain a special vertex, we define
$\A=(X_1)$, $\X'=(X_2,\dots,X_s)$, and $\Y' = \Y$. 

Case 2: If $s\geq1$, and $X_1$ contains at least one special vertex, there is a
unique integer $1\leq j\leq t$ such that $(X_1,Y_1,Y_2,\dots,Y_j)$ is a $k$-good
ordered forest. Then we define $\A=(X_1,Y_1,Y_2,\dots,Y_j)$,
$\X'=(X_2,X_3,\dots,X_s)$, and $\Y'=(Y_{j+1},Y_{j+2},\dots,Y_t)$. Since $\A$ is
a $k$-good ordered forest, there are $j$ special vertices in
$X_1,Y_1,Y_2,\dots,Y_j$. This implies that $(\X',\Y')$ is also a balanced pair.

Case 3: If $s=0$, then $\X=\emptyset$ and there are $t$ special vertices in
$Y_1,Y_2,\dots, Y_t$. Let $U=V(Y_1)\uplus\cdots\uplus V(Y_s)$ and let
$m=\max(U)$.  Suppose $Y_i$ contains $m$.  We apply the map $\phi$ in
Lemma~\ref{lem:CDG} to $Y_i$. Then $\phi(Y_i)$ is a rooted tree such that $m$ is
a leaf. If we apply the map $\rho$ in Lemma~\ref{lem:leaf} to $\phi(Y_i)$, we
get an ordered forest $\rho(\phi(Y_i)) = (T_1,T_2,\dots, T_\ell)$ on
$V(Y_i)\setminus\{m\}$. Let $\overline{\X}=(T_1,T_2,\dots, T_\ell)$ and
$\overline{\Y} = (Y_1,Y_2,\dots, \widehat{Y_i}, \dots,Y_t)$.  Note that the set
of vertices in $\overline{\X}$ and $\overline{\Y}$ is $U\setminus\{m\}$.  Let
$s_1<s_2<\cdots<s_t$ be the special vertices in $U$.  Suppose $U\setminus\{m\} =
\{ v_1<v_2<\cdots<v_{p}\}$ and $U\setminus\{s_i\} = \{
u_1<u_2<\cdots<u_{p}\}$. Then we define $\X'$ (resp.~$\Y'$) to be the ordered
forest obtained from $\overline{\X}$ (resp.~$\overline{\Y}$) by replacing $v_j$
with $u_j$ for all $j$. We define $\A$ to be the rooted tree with only one
vertex $s_i$.  It is clear from the construction that $\A$ is a $k$-good ordered
forest and $(\X',\Y')$ is a balanced pair.

In all cases, we clearly have $V(\X)\uplus V(\Y) = V(\A) \uplus V(\X') \uplus
V(\Y')$.  We now show that $f$ is invertible by constructing the inverse map
$g=f^{-1}$.  Suppose $\A=(A_1,\dots,A_r)$, $\X'=(X_1,\dots,X_s)$, and
$\Y'=(Y_1,\dots,Y_t)$, where $\A$ is a $k$-good forest and $(\X',\Y')$ is a
balanced pair.  Then we define $g(\A,(\X',\Y'))=(\X,\Y)$ as follows.

Case 1: If $r=1$, and $A_1$ does not have a special vertex, we define
$\X=(A_1,X_1,\dots,X_s)$ and $\Y=\Y'$.

Case 2: If $r\ge2$, we define $\X=(A_1,X_1,\dots,X_s)$ and
$\Y=(A_2,\dots,A_r, Y_1,\dots,Y_t)$. 

Case 3: If $r=1$, and $A_1$ has a special vertex, then by definition of $k$-good
forests, $A_1$ has only one vertex.  Let $U$ be the set of vertices in $\A$,
$\X'$, and $\Y'$, and $m=\max(U)$. Suppose $s_1<\cdots<s_{t+1}$ are the $t+1$
special vertices in $U$, and the unique vertex in $A_1$ is
$s_j$.  Let $U\setminus\{m\} = \{ v_1<v_2<\cdots<v_{p}\}$ and $U\setminus\{s_j\}
= \{ u_1<u_2<\cdots<u_{p}\}$. Then we define $\overline{\X}=(T_1,\dots,T_r)$ and
$\overline{\Y}=(R_1,\dots,R_s)$ to be the ordered forests obtained from $\X'$
and $\Y'$ by replacing $u_i$ with $v_i$ for all $i$.  Then the set of vertices
in $\X'$ and $\Y'$ is now $U\setminus\{m\}$. Thus we can construct the tree
$T=\rho^{-1}(\X')$ with maximum label $m$, and $R=\phi^{-1}(T)$ is a minimally
rooted tree. We define $\X=\emptyset$ and $\Y=(R_1,\dots,R_{i-1}, R,
R_i,\dots,R_s)$. 

It is easy to see that $g$ is the inverse map of $f$. 
\end{proof}

Now we are ready to define the map $\gamma$.  Suppose $(T_1,T_2,\dots,T_\ell,
T_{\ell+1}, T_{\ell+2}, \dots, T_{\ell+k})\in \BB_{n,k}$. Let
$\X=(T_1,T_2,\dots,T_\ell)$ and $\Y=(T_{\ell+1}, T_{\ell+2}, \dots,
T_{\ell+k})$. Then $(\X,\Y)$ is a balanced pair.  We define
$(\X_0,\Y_0),(\X_1,\Y_1),\dots,$ and $\A_1, \A_2,\dots,$ as follows.  Let
$(\X_0,\Y_0)=(\X,\Y)$. For $i\geq0$, if $(\X_i,\Y_i)\ne(\emptyset,\emptyset)$,
we define $\A_{i+1}, \X_{i+1}, \Y_{i+1}$ by $f(\X_i, \Y_i) = (\A_{i+1},
(\X_{i+1}, \Y_{i+1}))$. Let $p$ be the smallest integer such that
$\X_p=\Y_p=\emptyset$. Then we define $\gamma(\X,\Y)$ to be
$(\A_1,\A_2,\dots,\A_p)\in \CC_{n,k}$. Since $f$ is invertible, $\gamma$ is a
bijection.

\begin{example}
Let
\[
\F = \left(
\raisebox{.8cm}{
\pstree{\TR{10}}{\TR{4}} \comma
\TR{15} \comma
\pstree{\TR{1}}{\TR{3}}  \comma
\pstree{\TR{13}}{
  \TR{17}
}
\comma
    \pstree{\TR{8}}{
      \TR{19}
      \TR{18}
    }
\comma
\TR{16} \comma
\TR{12} \comma
\pstree{\TR{5}}{\TR{6}} \comma
\pstree{\TR{7}}{\pstree{\TR{14}}{\TR{9} \TR{2}}} \comma
\TR{11}
}
\right) \in \BB_{19,7}.
\]
Note that special vertices are less than or equal to $7$. Then
\[
\X=\X_0 = \left(
\raisebox{.8cm}{
\pstree{\TR{10}}{\TR{4}} \comma
\TR{15} \comma
\pstree{\TR{1}}{\TR{3}}
}
\right), \quad
\Y = \Y_0=\left(
\raisebox{.8cm}{
\pstree{\TR{13}}{
  \TR{17}
}
\comma
    \pstree{\TR{8}}{
      \TR{19}
      \TR{18}
    }
\comma
\TR{16} \comma
\TR{12} \comma
\pstree{\TR{5}}{\TR{6}} \comma
\pstree{\TR{7}}{\pstree{\TR{14}}{\TR{9} \TR{2}}} \comma
\TR{11}
}
\right),
\]
\[
\A_1 = \left(
\raisebox{.4cm}{
\pstree{\TR{10}}{\TR{4}} \comma
\pstree{\TR{13}}{
  \TR{17}
}
}
\right), \quad
\X_1 = \left(
\raisebox{.4cm}{
\TR{15} \comma
\pstree{\TR{1}}{\TR{3}}
}
\right), \quad
\Y_1=\left(
\raisebox{.8cm}{
   \pstree{\TR{8}}{
      \TR{19}
      \TR{18}
    }
\comma
\TR{16} \comma
\TR{12} \comma
\pstree{\TR{5}}{\TR{6}} \comma
\pstree{\TR{7}}{\pstree{\TR{14}}{\TR{9} \TR{2}}} \comma
\TR{11}
}
\right),
\]
\[
\A_2 = \left(
\raisebox{.1cm}{
\TR{15}}
\right), \quad
\X_2 = \left(
\raisebox{.4cm}{
\pstree{\TR{1}}{\TR{3}}
}
\right), \quad
\Y_2=\left(
\raisebox{.8cm}{
   \pstree{\TR{8}}{
      \TR{19}
      \TR{18}
    }
\comma
\TR{16} \comma
\TR{12} \comma
\pstree{\TR{5}}{\TR{6}} \comma
\pstree{\TR{7}}{\pstree{\TR{14}}{\TR{9} \TR{2}}} \comma
\TR{11}
}
\right),
\]
\[
\A_3 = \left(
\raisebox{.4cm}{
\pstree{\TR{1}}{\TR{3}}\comma
   \pstree{\TR{8}}{
      \TR{19}
      \TR{18}
    }
\comma
\TR{16}
}
\right), \quad
\X_3 = \emptyset, \quad
\Y_3=\left(
\raisebox{.8cm}{
\TR{12} \comma
\pstree{\TR{5}}{\TR{6}} \comma
\pstree{\TR{7}}{\pstree{\TR{14}}{\TR{9} \TR{2}}} \comma
\TR{11}
}
\right).
\]
In $\Y_3$ the largest vertex $14$ is in the third tree. Using $\phi$ and $\rho$
we get
\[
(\rho \circ \phi) \left(
\raisebox{.8cm}{
\pstree{\TR{7}}{\pstree{\TR{14}}{\TR{9} \TR{2}}} 
}
\right) = 
\rho \left(
\raisebox{.8cm}{
 \pstree{\TR{9}}{
    \pstree{\TR{7}}{
      \TR{14}
    }
    \TR{2}
  }
}
\right) = 
\left(
\raisebox{.4cm}{
\pstree{\TR{9}}{\TR{2}} \comma 
\TR{7}
}
\right). 
\]
Thus
\[
\overline{\X_3} = 
\left(
\raisebox{.4cm}{
\pstree{\TR{9}}{\TR{2}} \comma 
\TR{7}
}
\right), \quad
\overline{\Y_3}=\left(
\raisebox{.4cm}{
\TR{12} \comma
\pstree{\TR{5}}{\TR{6}} \comma
\TR{11}
}
\right).
\]
Since $\X_3$ and $\Y_3$ have 4 special vertices $2,5,6,7$, and $6$ is the third
smallest special vertex, we replace the vertices in 
$U\setminus\{14\}$ with those in $U\setminus\{6\}$. Since
\[
\begin{array}{ccccccccccc}
U\setminus\{14\} &=&    \{ & 2, & 5, & 6, & 7, & 9, & 11, & 12 &\}, \\
U\setminus\{6\} &=&    \{ &2, & 5, & 7, & 9, & 11, & 12, & 14 &\},
\end{array}
\]
we get
\[
\A_4 = \left(
\raisebox{.1cm}{
\TR{6}
}
\right), \quad
\X_4 = \overline{\X_3'}= 
\left(
\raisebox{.4cm}{
\pstree{\TR{11}}{\TR{2}} \comma 
\TR{9}
}
\right), \quad
\Y_4=\overline{\Y_3'}= \left(
\raisebox{.4cm}{
\TR{14} \comma
\pstree{\TR{5}}{\TR{7}} \comma
\TR{12}
}
\right),
\]
\[
\A_5 = \left(
\raisebox{.4cm}{
\pstree{\TR{11}}{\TR{2}} \comma
\TR{14} 
}
\right), \quad
\X_5=
\left(
\raisebox{.1cm}{
\TR{9}
}
\right), \quad
\Y_5=
\left(
\raisebox{.4cm}{
\pstree{\TR{5}}{\TR{7}} \comma
\TR{12}
}
\right),
\]
\[
\A_6 = \left(
\raisebox{.1cm}{
\TR{9}
}
\right), \quad
\X_6=\emptyset, \quad
\Y_6=
\left(
\raisebox{.4cm}{
\pstree{\TR{5}}{\TR{7}} \comma
\TR{12}
}
\right).
\]
In $\Y_6$, the largest vertex $12$ is in the second tree. 
\[
(\rho\circ\phi)(12) = \rho(12) = \emptyset.
\]
Thus
\[
\overline{\X_6}= \emptyset, \quad
\overline{\Y_6}=
\left(
\raisebox{.4cm}{
\pstree{\TR{5}}{\TR{7}}
}
\right).
\]
If we replace the labels in $\{5,7\}$ with $\{5,12\}$ we get
\[
\A_7 = \left(
\raisebox{.1cm}{
\TR{7}
}
\right), \quad
\X_7 =\emptyset, \quad
\Y_7=
\left(
\raisebox{.4cm}{
\pstree{\TR{5}}{\TR{12}}
}
\right).
\]
Since
\[
(\rho\circ\phi)\left(\raisebox{.4cm}{\pstree{\TR{5}}{\TR{12}}}\right) = 
\rho\left(\raisebox{.4cm}{\pstree{\TR{5}}{\TR{12}}}\right) = 5, 
\]
we have $\overline{\X_7}= (5)$ and $\overline{\Y_6}= \emptyset$. 
Replacing $5$ with $12$ we get
\[
\A_8 = \left(
\raisebox{.1cm}{
\TR{5}
}
\right), \quad
\X_8 = 
\left(
\raisebox{.1cm}{
\TR{12}
}
\right), \quad
\Y_8=\emptyset.
\]
Finally we get
\[
\A_9 = \left(
\raisebox{.1cm}{
\TR{12}
}
\right), \quad
\X_9 =\emptyset, \quad
\Y_9=\emptyset.
\]
Thus 
\[
\gamma(\F)=
\left(
\left(
\raisebox{.4cm}{
\pstree{\TR{10}}{\TR{4}} \comma
\pstree{\TR{13}}{\TR{17}}}
\right), 
\left(
\raisebox{.1cm}{
\TR{15}}
\right), 
\left(
\raisebox{.4cm}{
\pstree{\TR{1}}{\TR{3}}\comma
   \pstree{\TR{8}}{
      \TR{19}
      \TR{18}
    }
\comma
\TR{16}
}
\right), 
\left(
\raisebox{.1cm}{
\TR{6}
}
\right), 
\left(
\raisebox{.4cm}{
\pstree{\TR{11}}{\TR{2}} \comma
\TR{14} 
}
\right), 
\left(
\raisebox{.1cm}{
\TR{9}
}
\right), 
\left(
\raisebox{.1cm}{
\TR{7}
}
\right), 
\left(
\raisebox{.1cm}{
\TR{5}
}
\right), 
\left(
\raisebox{.1cm}{
\TR{12}
}
\right)
\right)\in \CC_{19,7}.
\]

\end{example}


\subsection{A bijection $\zeta:\CC_{n,k}\to\FF_{n,k}$}

Recall that a special vertex is a vertex with label at most $k$.

\begin{lem}\label{lem:good}
  For a fixed set $A\subset [n]$ with $|A|\geq2$, there is a bijection $\psi$
  from the set of $k$-good ordered forests on $A$ to the set of rooted trees on
  $A$ whose special vertices are nonleaves.
\end{lem}
\begin{proof}
  Suppose $\F=(A_1,A_2,\dots,A_p)$ is a $k$-good ordered forest.  We first set
  all special vertices in $\F$ \emph{active}. Find the smallest vertex $v$ among
  the active vertices with minimal distance from the root in $A_1$. Then
  exchange the subtrees attached to $v$ and those attached to the root $r$ of
  $A_2$, and then attach the resulting tree rooted at $r$ to $v$ as shown below.
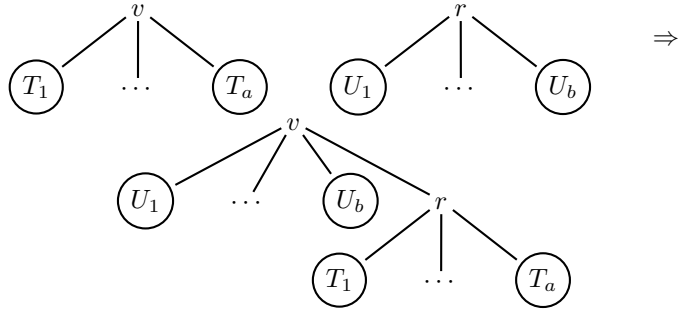
\begin{center}
\psset{nodesep=2pt, levelsep=30pt}
\pstree{\TR{$v$}}{
  \Tcircle{$T_1$}
 \TR{$\cdots$}
  \Tcircle{$T_a$}
}
\qquad
\pstree{\TR{$r$}}{
  \Tcircle{$U_1$}
 \TR{$\cdots$}
  \Tcircle{$U_b$}
}
\qquad
\raisebox{-.5cm}{$\Rightarrow$}
\qquad
\pstree{\TR{$v$}}{
  \Tcircle{$U_1$}
 \TR{$\cdots$}
  \Tcircle{$U_b$}
  \pstree{\TR{$r$}}{
    \Tcircle{$T_1$}
    \TR{$\cdots$}
    \Tcircle{$T_a$}
  }
}
\end{center}
We then make $v$ inactive. Note that $v$ is a nonleaf after this procedure. We
do the same thing with the resulting tree and $A_3$, and proceed until there are
no active special vertices. Since $(A_1,A_2,\dots,A_p)$ is $k$-good, we can
eventually combine all of $A_1,A_2,\dots,A_p$ into a single rooted tree in which
the special vertices are nonleaves. We define $\psi(\F)$ to be the resulting
tree. It is straightforward to check that $\psi$ is invertible. 
\end{proof}

Let $(\F_1,\F_2,\dots,\F_h)\in \CC_{n,k}$. For each $k$-good forest $\F_i$ we
define $T_i=\psi(\F_i)$ if $\F_i$ has at least $2$ vertices, and $T_i=X$ if
$\F_i=(X)$ and $X$ has only one vertex. Then we define
$\zeta(\F_1,\F_2,\dots,\F_h)=\rho^{-1}(T_1,\dots,T_h)$. Since $\rho^{-1}$ and
$\psi$ are invertible, $\zeta$ is a bijection.

\begin{example}
Let $(\F_1,\F_2,\dots,\F_h)$ be the following:
\[
\left(
\left(
\raisebox{.4cm}{
\pstree{\TR{10}}{\TR{4}} \comma
\pstree{\TR{13}}{\TR{17}}}
\right), 
\left(
\raisebox{.1cm}{
\TR{15}}
\right), 
\left(
\raisebox{.4cm}{
\pstree{\TR{1}}{\TR{3}}\comma
   \pstree{\TR{8}}{
      \TR{19}
      \TR{18}
    }
\comma
\TR{16}
}
\right), 
\left(
\raisebox{.1cm}{
\TR{6}
}
\right), 
\left(
\raisebox{.4cm}{
\pstree{\TR{11}}{\TR{2}} \comma
\TR{14} 
}
\right), 
\left(
\raisebox{.1cm}{
\TR{9}
}
\right), 
\left(
\raisebox{.1cm}{
\TR{7}
}
\right), 
\left(
\raisebox{.1cm}{
\TR{5}
}
\right), 
\left(
\raisebox{.1cm}{
\TR{12}
}
\right)
\right)\in \CC_{19,7}.
\]
Then the map $\psi$ sends
\[
\left(
\raisebox{.4cm}{
\pstree{\TR{10}}{\TR{4}} \comma
\pstree{\TR{13}}{\TR{17}}}
\right)
\mapsto 
\raisebox{.4cm}{
\pstree{\TR{10}}{
    \pstree{\TR{4}}{
      \TR{17}
      \TR{13}
    }
  }
},
\]
\[
\left(
\raisebox{.4cm}{
\pstree{\TR{1}}{\TR{3}}\comma
   \pstree{\TR{8}}{
      \TR{19}
      \TR{18}
    }
\comma
\TR{16}
}
\right)
\mapsto
\left(
\raisebox{.8cm}{
\pstree{\TR{1}}{
  \TR{19} \TR{18}
  \pstree{\TR{8}}{
    \TR{3}
  }
}\comma
\TR{16}
}
\right)
\mapsto
\raisebox{1.2cm}{
\pstree{\TR{1}}{
  \TR{19} \TR{18}
  \pstree{\TR{8}}{
    \pstree{\TR{3}}{
      \TR{16}
    }
  }
}
},
\]
\[
\left(
\raisebox{.4cm}{
\pstree{\TR{11}}{\TR{2}} \comma
\TR{14} 
}
\right)
\mapsto
\raisebox{.4cm}{
\pstree{\TR{11}}{\pstree{\TR{2}}{\TR{14}}}
}.
\]  
Thus we obtain $(T_1,\dots,T_h)$:
\[
\left(
\raisebox{1.1cm}{
\pstree{\TR{10}}{
    \pstree{\TR{4}}{
      \TR{17}
      \TR{13}
    }
  } \comma 
\TR{15} \comma 
\pstree{\TR{1}}{
  \TR{19} \TR{18}
  \pstree{\TR{8}}{
    \pstree{\TR{3}}{
      \TR{16}
    }
  }
} \comma
\TR{6} \comma 
\pstree{\TR{11}}{\pstree{\TR{2}}{\TR{14}}}\comma
\TR{9} \comma 
\TR{7} \comma 
\TR{5} \comma 
\TR{12} 
}
\right)
\]
If we add the vertex $n+1$, we obtain $\zeta(\F_1,\F_2,\dots,\F_h)$:
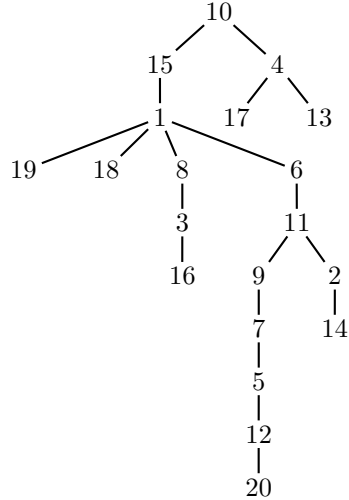
\[
\pstree{\TR{10}}{
 \pstree{\TR{15}}{
    \pstree{\TR{1}}{
      \TR{19} \TR{18}
      \pstree{\TR{8}}{
        \pstree{\TR{3}}{
          \TR{16}
        }
     }
      \pstree{\TR{6}}{
        \pstree{\TR{11}}{
         \pstree{\TR{9}}{
            \pstree{\TR{7}}{
              \pstree{\TR{5}}{
                \pstree{\TR{12}}{
                  \TR{20}
                }
              }
            }
          }
          \pstree{\TR{2}}{
            \TR{14}
          }
       }
      }
    }
  } 
  \pstree{\TR{4}}{
    \TR{17}
    \TR{13}
  }
}
\]
\end{example}

\section{Some properties of $|\TT_{n,k}|$}\label{sec:some-properties-tt_n}

We denote the cardinality of $|\TT_{n,k}|$ by $a_{n,k}$. 
In \cite{SeoShin} Seo and Shin proved the following.
\begin{thm}\cite{SeoShin}
We have
 \begin{align}
    \label{eq:3}
    \sum_{k\ge0} \binom{m+k}{k} a_{n,k} &= (n+m+1)^n, \\
\label{eq:4}
    \sum_{k\ge0} \frac{1}{k} a_{n,k} &= n^n.
 \end{align}
\end{thm}
We give another proof using generating functions.

\begin{proof}
By Theorem~\ref{thm:SeoShin}, $a_{n,k}$ equals the number of  words of length
$n$ on $[n]$ with at least one $i$ for all $i\in[k]$. Thus
\[
a_{n,k} = \left[ \frac{x^n}{n!}\right] (e^x-1)^k e^{(n-k)x}
= \left[ \frac{x^n}{n!} \right] (1-e^{-x})^k e^{nx},
\]
where $\left[ y^n\right]f(y)$ denotes the coefficient of $y^n$ in $f(y)$. 
Then we have
\begin{align*}
\sum_{k\ge0} \binom{m+k}{k} a_{n,k} & =   
\left[ \frac{x^n}{n!}\right] e^{nx} \sum_{k\ge0}  \binom{m+k}{k}  (1-e^{-x})^k \\
&= \left[ \frac{x^n}{n!}\right] e^{nx} \frac{1}{(1-(1-e^{-x}))^{m+1}}\\
&= \left[ \frac{x^n}{n!}\right] e^{(n+m+1)x} = (n+m+1)^n,
\end{align*}
where the following binomial theorem \cite[(1.20)]{EC1second} is used:
\[
\frac{1}{(1-x)^n} = \sum_{k\geq0} \binom{n+k-1}{k} x^k.
\]

The second identity is proved similarly:
\begin{align*}
 \sum_{k\ge0} \frac{1}{k} a_{n,k} 
&=\left[ \frac{x^n}{n!}\right] e^{nx} \sum_{k\ge0} \frac{(1-e^{-x})^k}{k}\\
&=\left[ \frac{x^n}{n!}\right] e^{nx} \ln \frac{1}{1-(1-e^{-x})}\\
&=\left[ \frac{x^n}{n!}\right] x e^{nx}
= n! \left[ x^{n-1}\right] e^{nx} = n!\frac{n^{n-1}}{(n-1)!} = n^n.
\end{align*}
\end{proof}

In \cite{SeoShin} they asked for a bijective proof of \eqref{eq:4}. We can prove
it bijectively using our bijections as follows.

\begin{proof}[Bijective proof of \eqref{eq:4}]
By Theorem~\ref{thm:main}, $a_{n,k}$ is also equal to $|\BB_{n,k}|$, the number
of ordered forests $(T_1,T_2,\dots, T_\ell, T_{\ell+1}, \dots,T_{\ell+k})$ on
$[n]$ such that $T_{\ell+1}, \dots,T_{\ell+k}$ are minimally rooted. Thus
$a_{n,k}/k$ is equal to the number of pairs $(\F, C)$ of an ordered forest
$\F=(T_1,T_2,\dots, T_\ell)$ and a cycle $C=[T_{\ell+1}, \dots,T_{\ell+k}]$ of
$k$ minimally rooted trees such that the vertex sets of $T_1,\dots, T_{\ell+k}$
form a partition of $[n]$. Then, by Theorem~\ref{thm:main}, the number of cycles
of $k$ minimally rooted trees whose vertex sets form a subset $A$ of $[n]$ is
equal to the set of rooted trees $T$ on $A$ with $|\MD(T)|=k$. Thus $a_{n,k}/k$
is equal to the number of ordered forests $(T_1,T_2,\dots, T_\ell, T)$ on $[n]$
with $|\MD(T)|=k$. The sum of $a_{n,k}/k$ for all $k$ is equal to the number of
ordered forests on $[n]$. Suppose $(T_1,T_2,\dots, T_\ell)$ is an ordered forest
on $[n]$ and $r_i$ is the root of $T_i$ for $i\in[\ell]$. By adding the edges
$r_1r_2,r_2r_3,\dots, r_{\ell-1}r_\ell$, we get a rooted tree, say $H$. If we
know the root $r_\ell$ of the last tree, then we can recover the ordered forest
$(T_1,T_2,\dots, T_\ell)$ from $H$. Thus there is a bijection between the set of
ordered forests on $[n]$ and the set of rooted trees on $[n]$ with a choice of
$r_\ell$. The latter set has cardinality $n^n$ by Pr\"ufer code. This proves
\eqref{eq:4}. 
\end{proof}

\section{Another proof of Theorem~\ref{thm:CDG}}
\label{sec:anoth-proof-theor}

Using Pr\"ufer code one can easily see that $\binom{n}{\ell} n^{n-\ell}$ is
equal to the number of rooted trees on $\{0,1,2,\dots,n+1\}$ such that $0$ is
the root with $\ell+1$ children and $n+1$ is a leaf. By deleting the root $0$,
such a tree is identified with a forest on $[n+1]$ with $\ell+1$ rooted trees such that
$n+1$ is a leaf. Thus by replacing $n+1$ with $n$, we can rewrite
Theorem~\ref{thm:CDG} as follows.

\begin{thm}\cite{Chauve2000}\label{thm:children}
  The number of rooted trees of $[n]$ such that the root has $\ell$ smaller children
  equals the number of forests on $[n]$ with $\ell+1$ trees such that $n$ is a leaf.
\end{thm}
\begin{proof}
  Let $T$ be a rooted trees of $[n]$ such that the root has $\ell$ smaller
  children. We will construct a forest corresponding to $T$.  Recall the
  bijection $\alpha:\TT_{n,k}\to\AA_{n,k}$.  Suppose $T\in \TT_{n-1,k}$,
  $\alpha(T)=[T_0,T_1,\dots,T_k]$, $r_i$ is the root of $T_i$ for
  $i=0,1,2,\dots,k$. By shifting cyclically we can assume that $r_0$ is the
  largest root. Note that, by the construction of $\alpha$, $T$ is rooted at
  $r_0$.  Also, from the construction of $\alpha$, it is easy to see that the
  smaller children of the root $r_0$ in $T$ are exactly the left-to-right maxima
  of $r_1,r_2,\dots,r_k$.  Suppose $r_{i_1}<r_{i_2}<\cdots<r_{i_\ell}$ are the
  smaller children of $r_0$ in $T$. Then $1=i_1<i_2<\cdots<i_\ell\leq k$.
  Suppose $n$ is contained in $T_j$. Let $T_1', T_2',\dots, T_k'$ be the
  arrangement of the trees $T_0,T_1,\dots,\widehat{T_j}, \dots, T_k$ such that
  the word $r_1'r_2'\cdots r_k'$ of the roots of $T_1',\dots,T_k'$ are
  order-isomorphic to the word $r_1r_2\cdots r_k$. Notice that $r_{i_1}',
  r_{i_2}', \dots,r_{i_\ell}'$ are the left-to-right maxima of $r_1'r_2'\cdots
  r_k'$. Thus the following map is invertible:
  \begin{equation}
    \label{eq:5}
(T_1', T_2',\dots, T_k') \mapsto
\{[T_{i_1}', \dots,T_{i_2-1}'], [T_{i_2}', \dots, T_{i_3-1}'], \dots,
[T_{i_\ell}', \dots, T_k']\}.
 \end{equation}
 Now we apply the inverse map $\alpha^{-1}$ of the bijection $\alpha$ to each
 cycle in \eqref{eq:5}. Then we get a set of rooted trees. Together with $T_j$,
 we obtain a forest on $[n]$. Since $T_j$ is the tree containing $n$, we can
 recover the original tree $T$ from the forest. This gives a bijection between
 the two sets in the theorem.
\end{proof}


The proof of Theorem~\ref{thm:children}, in fact, gives a generalization as
follows.

\begin{cor}
  Let $A(n,\ell,k)$ denote the number of rooted trees of $[n]$ such that the
  root has $\ell$ smaller children and the minimal decreasing subtree has $k+1$
  vertices. Let $B(n,\ell,k)$ denote the number of forests on $[n]$ with
  $\ell+1$ trees such that $n$ is a leaf, and the sum of $|\MD(T)|$ for all
  trees $T$ in the forest except the one containing $n$ is equal to $k$. Then
  \[
A(n,\ell,k) = B(n,\ell,k).
\]
\end{cor}
\begin{proof}
  This can be checked by the following observation.  Consider a cycle $C$ in
  \eqref{eq:5}, and $T=\alpha^{-1}(C)$. Then $|\MD(T)|$ is the number of trees
  in $C$. Thus the sum of $|\MD(T)|$ for all cycles $C$ in \eqref{eq:5} is $k$.
\end{proof}

\section*{Acknowledgement}
The author would like to thank Dennis Stanton for helpful discussion and
comments, especially for the idea in the proof using generating functions in
Section~\ref{sec:some-properties-tt_n}.


\end{document}